\documentclass{article}

\usepackage{amsmath,amssymb,amsfonts}
\usepackage{amsthm} 
\usepackage[breaklinks=true]{hyperref}
\usepackage{xspace,enumerate}
\usepackage{booktabs}
\usepackage{color,epsfig}
\usepackage{burdges}
\usepackage{calc}

\newcommand\Uir{U_{0,r}}

\begin{document}

\title{On Frattini arguments in $L$-groups \\ of finite Morley rank}
\author{Jeffrey Burdges%
\thanks{Supported by NSF grant DMS-0100794, and
Deutsche Forschungsgemeinschaft grant Te 242/3-1.}}

\maketitle

\begin{proposition}\label{p:invCarter}
Let $\hat{G}$ be a group of finite Morley rank,
 let $G$ itself be a definable connected normal subgroup group of $\hat{G}$,
 and let $S$ be a Sylow 2-subgroup of $\hat{G}$.
Then $G$ has an $S$-invariant Carter subgroup.
\end{proposition}

%
%

We use the following facts,
 as well as the \Frecon-Jaligot construction
 and the conjugacy of descent tori.

\begin{fact}[{\cite[\qCorollary 3.5]{Bu05a}}]\label{nildecomp}
Let $H$ be a nilpotent group of finite Morley rank.
Then $H = D * B$ is a central product of definable characteristic
 subgroups $D,B \leq H$ where $D$ is divisible and
 $B$ has bounded exponent (which is connected iff $H$ is connected).
Let $T$ be the torsion part of $D$.
Then we have decompositions of $D$ and $B$ as follows.
\begin{eqnarray*}
D &=& d(T) * U_{0,1}(H) * U_{0,2}(H) * \cdots \\
B &=& U_2(H) \oplus U_3(H) \oplus U_5(H) \oplus \cdots
\end{eqnarray*}
\end{fact}

\begin{fact}[{\cite[\qLemma 4.4]{Bu05a}}]\label{Uz_on_Up}
Let $H$ be a connected solvable group of finite Morley rank.
Suppose that $S$ is a nilpotent $\Uir$-subgroup of $H$, and
 that $H = U_p(H) S$ for some $p$ prime.
Then $H$ is nilpotent, and $[U_p(H),S]=1$.
\end{fact}

\begin{fact}[{\cite[\qLemma 4.4]{Bu05a}}]\label{nilpotencepre2}
Let $H = K T$ be a group of finite Morley rank with
 $K \normal H$ a nilpotent $U_{0,r}$-group and
  $T$ a nilpotent $U_{0,s}$-group for some $s \geq r$.
Then $H$ is nilpotent.
\end{fact}

\smallskip

\begin{proof}[Proof of Proposition \ref{p:invCarter}]
We take $G$ to be a minimal counterexample, i.e.\ proper
 connected $S$-invariant subgroups of $G$ have $S$-invariant Carter subgroups.
Then $G$ is nonsolvable by a Frattini argument.
We observe that Carter subgroups of $G$ correspond to
 Carter subgroups of $G/Z^\o(G)$.  
So $G$ is centerless too.
We may also assume that $C_S(G) = 1$.

If $G$ contains divisible torsion,  we take $T$ to be
 a maximal descent torus of $G$ (exists by \cite{Ch05}).
By \cite{Ch05}, all maximal descent tori are conjugate,
 so $G = N_G(T) G$ by a Frattini argument.
Hence a conjugate or $T$ is $S$-invariant.
There is a Carter subgroup $Q$ containing $T$,
 by a \Frecon-Jaligot construction.
By Fact \cite[\qTheorem 6.16]{BN}, $T$ is central in $M := N^\o_G(T)$.
Now we have $M < G$, since $G$ is centerless.
Since $M$ is $S$-invariant,
 there is an $S$-invariant Carter subgroup $Q$ of $M$.
Since $T$ is a characteristic subgroup of $Q$,
 $N^\o_G(Q) \leq N^\o_G(T) \leq M$.
So $Q$ is a Carter subgroup of $G$.
Hence we may assume that $G$ has no divisible torsion.

We find a maximal $S$-invariant group $R$
 with ``minimal degree of unipotence'' $r$.
Let $r$ be the minimal reduced rank such that
 there is an $S$-invariant nilpotent $\Uir$-subgroup of $G$,
or take $r = \infty$ if there are no $S$-invariant nilpotent $\Uir$-subgroups.
Let $i$ be a central involution of $S$.
Since $G$ is nonabelian,
 $C^\o_G(i)$ is infinite by \cite[Ex. 13 \& 15 on p.~79]{BN}.
and clearly $S$-invariant.
Since $C^\o_G(i) < G$,
 $C^\o_G(i)$ has a definable $S$-invariant nilpotent subgroup,
 by induction.
So, if $r = \infty$, there is an $S$-invariant nilpotent group
 of bounded exponent, by Fact \ref{nildecomp}.
Let $R$ be a maximal $S$-invariant nilpotent $\Uir$-subgroup of $G$,
 or else a maximal $S$-invariant $p$-unipotent subgroup of $G$ if $r = \infty$.
Now consider the $S$-invariant group $M := N^\o_G(R)$.

First suppose that $M < G$.
Then there is an $S$-invariant Carter subgroup $Q$ of $M$,
 by induction.
Now $R Q$ is solvable.
We show that $R$ is characteristic in $Q$.
If $r = \infty$,
 $R Q$ is nilpotent by Fact \ref{Uz_on_Up}, 
 and $R Q$ has bounded exponent.
So $R = Q$ by maximality.
If $r < \infty$,
 $r$ is the minimal reduced rank found in $Q$,
  since $U_{0,t}(Q)$ is $S$-invariant for all $t$.
So $R Q$ is nilpotent by Fact \ref{nilpotencepre2}.
 $R \leq Q$, and hence $\Uir(Q) = R$.
In either case, $N^\o_G(Q) \leq M$.
So $Q$ is a Carter subgroup of $G$.

Next suppose that $M = G$, i.e.\ $R \normal G$.
Then there is an $S$-invariant Carter subgroup $\bar{Q}$ of $\bar{G} = G/R$.
Now the pullback $Q$ of $\bar{Q}$ to $G$ is solvable,
 and $R \leq Q$.
Since $G$ is nonsolvable, $\bar{Q} < \bar{G}$, and $Q < G$.
Since $Q$ is $S$-invariant,
 there is an $S$-invariant Carter subgroup $L$ of $Q$,
 by induction.
We show that $Q = R L$ is nilpotent.
If $r = \infty$,
 $R L$ is nilpotent by Fact \ref{Uz_on_Up}. 
If $r < \infty$,
 $r$ is the minimal reduced rank found in $L$ (or $Q$),
 since $U_{0,t}(L)$ is $S$-invariant for all $t$.
So $R L$ is nilpotent by Fact \ref{nilpotencepre2},
 and $R \leq L$.
By maximality of $R$, $\Uir(L) = R$.
By Fact \ref{nildecomp}.
 $U_{0,t}(L)$ centralizes $R$ for $t > r$.
In either case, $Q = R L$ is nilpotent.
We observe that $N_G(Q) / R \leq N_{\bar{G}}(\bar{Q})$
 has the same rank as $\bar{Q}$.
So $\rk(N_G(Q)) = \rk(Q)$, and $N^\o_G(Q) = Q$.
Thus $Q$ is an $S$-invariant Carter subgroup of $G$.
\end{proof}

\small
\bibliographystyle{alpha} 
\bibliography{burdges,fMr}

\end{document}